\RequirePackage{fix-cm}
\documentclass[smallextended]{svjour3}       
\smartqed  
\usepackage{graphicx}
\usepackage{amssymb}
\usepackage{amsmath}

\usepackage{dsfont}
\usepackage{enumerate}
\begin{document}

\title{Closed form expressions for Appell polynomials}


\author{Jos\'{e} A. Adell         \and
        Alberto Lekuona
}


\institute{Jos\'{e} A. Adell \at
              Departamento de M\'{e}todos Estad\'{\i}sticos, Facultad de
Ciencias, Universidad de Zaragoza, 50009 Zaragoza (Spain)\\
              \email{adell@unizar.es}           
           \and
           Alberto Lekuona \at
             Departamento de M\'{e}todos Estad\'{\i}sticos, Facultad de
Ciencias, Universidad de Zaragoza, 50009 Zaragoza (Spain)\\
              \email{lekuona@unizar.es}
}

\date{Received: date / Accepted: date}

\maketitle

\begin{abstract}

We show that any Appell sequence can be written in closed form as a forward difference transformation of the identity. Such transformations are actually multipliers in the abelian group of the Appell polynomials endowed with the operation of binomial convolution. As a consequence, we obtain explicit expressions for higher order convolution identities referring to various kinds of Appell polynomials in terms of the Stirling numbers. Applications of the preceding results to generalized Bernoulli and Apostol-Euler polynomials of real order are discussed in detail.

\end{abstract}

\section{Introduction}\label{s1}

Given an arbitrary function $f:\mathds{R}\to \mathds{R}$, the usual $k$th forward difference of $f$ is defined as
\begin{equation*}
\Delta^kf(x)=\sum_{j=0}^k \binom{k}{j}(-1)^{k-j} f(x+j),\quad x\in \mathds{R},\quad k=0,1,2,\ldots .
\end{equation*}
For any $x\in \mathds{R}$ and $n=0,1,2,\ldots $, denote by $B_n(x)$ and $E_n(x)$ the classical Bernoulli and Euler polynomials of degree $n$, respectively, and by $I_n(x)=x^n$. The starting point of this work are the following well known explicit expressions of such polynomials
\begin{equation}\label{eq1.1}
B_n(x)=\sum_{k=0}^n \dfrac{(-1)^k}{k+1}\Delta^k I_n(x),\qquad E_n(x)=\sum_{k=0}^n \dfrac{(-1)^k}{2^k}\Delta^k I_n(x).
\end{equation}
Closed form expressions for other Appell polynomials, usually in terms of the Stirling numbers of the second kind, are known in the literature. Among others, we mention the explicit expressions for the generalized Bernoulli polynomials of complex order (c.f. Todorov \cite{Tod1985}, Srivastava and Todorov \cite{SriTod1988}, and Guo and Qi \cite{GuoQi2014}), the generalized Apostol-Bernoulli polynomials of integer order (c.f. Luo and Srivastava \cite{LuoSri2005}), the generalized Apostol-Euler polynomials of real order (c.f. Luo \cite{Luo2006}), and the generalized Apostol-Bernoulli and Apostol-Euler polynomials of integer order (c.f. Kim \textit{et al.} \cite{KimKimDolRim2013}).

The aim of this paper is twofold. In first place, to show that any Appell sequence $A(x)=(A_n(x))_{n\geq 0}$ can be written in closed form as
\begin{equation}\label{eq1.0.3}
  A_n(x)=\sum_{k=0}^n \dfrac{a_k}{k!}\Delta^k I_n(x),\qquad n=0,1,\ldots ,\quad x\in \mathds{R},
\end{equation}
where the sequence of real numbers $\boldsymbol{a}=(a_n)_{n\geq 0}$ is unique and, for this reason, may be called the sequence associated to $A(x)$. This is done in section~\ref{s5}. In a recent paper, Boutiche \textit{et al.} \cite{BouRahSri2017} define the generalized Stirling numbers $S_n^k(x)$ as
\begin{equation}\label{eq1.0.4}
    S_n^k(x)=\dfrac{\Delta^k I_n(x)}{k!},\quad n,k=0,1,\ldots ,\quad k\leq n,\ x\in \mathds{R},
\end{equation}
and show that the generalized Bernoulli and Euler polynomials of complex order can be written in terms of such numbers. In this respect, our results in Section~\ref{s5}, particularly Theorem~\ref{th5.7}, appear as generalizations of the results in \cite{BouRahSri2017}.

Of course, the main problem is to give an explicit expression of the sequence $\boldsymbol{a}=(a_n)_{n\geq 0}$ as simple as possible. In this regard, we have included three preliminary sections. Section~\ref{s2} contains some basic results on binomial convolutions of Appell sequences already shown in \cite{AdeLek2017.2}. In Section~\ref{s3}, we introduce the notion of forward difference transformation of an Appell sequence, so that formula \eqref{eq1.0.3} states that any Appell sequence $A(x)$ is a forward difference transformation  of the identity. We have also included in Section~\ref{s4} some properties of the Stirling numbers, because the sequence $\boldsymbol{a}=(a_n)_{n\geq 0}$ may be described in terms of them. Finally, applications of formula \eqref{eq1.0.3} to generalized Bernoulli and Apostol-Euler polynomials of real order are given in Theorem~\ref{th6.1} and Theorem~\ref{th7.1} in sections~\ref{s6} and \ref{s7}, respectively.

In second place, to show that formula \eqref{eq1.0.3} can be used to obtain explicit expression for general convolution identities referring to various kinds of Appell polynomials. More precisely, we give formulas in closed form for
\begin{equation}\label{eq1.0.1}
  \sum_{j_1+\cdots +j_m=n}\dfrac{n!}{j_1!\cdots j_m!} A_{j_1}(x_1)\cdots A_{j_m}(x_m),\qquad n,m=1,\ldots ,
\end{equation}
where $A(x)$ are the Bernoulli or the Apostol-Euler polynomials (see Theorem~\ref{th6.11}, Theorem~\ref{th7.13}, and Theorem~\ref{th7.14} in sections~\ref{s6} and \ref{s7}).

In this respect, a classical identity of second order concerning the Bernoulli polynomials, which goes back to N\"orlund \cite{Nor1922}, is the following
\begin{equation}\label{eq1.0.2}
  \sum_{k=0}^n \binom{n}{k}B_k(x)B_{n-k}(y)=-n(x+y-1)B_{n-1}(x+y)-(n-1)B_n (x+y),
\end{equation}
where $n=1,2,\ldots $. Since the pioneering work by Dilcher \cite{Dil1996}, many authors have provided different expressions for the sums in \eqref{eq1.0.1}. See, for instance, Wang \cite{Wan2013}, Agoh and Dilcher \cite{AgoDil2014}, He and Araci \cite{HeAra2014}, Wu and Pan \cite{WuPan2014}, He \cite{He2017}, Dilcher and Vignat \cite{DilVig2016}, and the references therein. As it happens in \eqref{eq1.0.2}, a common feature of the identities for the sums in \eqref{eq1.0.1} usually found in the literature is that the right-hand side of such identities contains the polynomials $A(x)$ themselves. This can be avoided in many usual cases, particularly, when $A(x)$ are the Bernoulli and the Apostol-Euler polynomials, as shown in Sections~\ref{s6} and \ref{s7}.

\section{Binomial convolution of Appell sequences}\label{s2}

Let $\mathds{N}$ be the set of  positive integers and $\mathds{N}_0=\mathds{N}\cup \{0\}$. Unless otherwise specified, we assume from now on that $n\in \mathds{N}_0$, $x\in \mathds{R}$, and $z\in \mathds{C}$ with $|z|\leq r$, where $r>0$ may change from line to line. Denote by $\mathcal{G}$ the set of all real sequences $\boldsymbol{u}=(u_n)_{n\geq 0}$ such that $u_0\neq 0$ and
\begin{equation*}
\sum_{n=0}^\infty |u_n|\dfrac{r^n}{n!}<\infty,
\end{equation*}
for some radius $r>0$. If $\boldsymbol{u}\in \mathcal{G}$, we denote its generating function by
\begin{equation*}
G(\boldsymbol{u},z)=\sum_{n=0}^\infty u_n \dfrac{z^n}{n!}.
\end{equation*}
If $\boldsymbol{u},\boldsymbol{v}\in \mathcal{G}$, the binomial convolution of $\boldsymbol{u}$ and $\boldsymbol{v}$, denoted by $\boldsymbol{u}\times \boldsymbol{v}=((u\times v)_n)_{n\geq 0}$, is defined as
\begin{equation}\label{eq2.1}
(u\times v)_n=\sum_{k=0}^n \binom{n}{k}u_k v_{n-k}.
\end{equation}
From a computational point of view, the interest of this definition relies on the following characterization in terms of generating functions (c.f. \cite[Proposition 2.1]{AdeLek2017.2})
\begin{equation}\label{eq2.2}
G(\boldsymbol{u}\times \boldsymbol{v},z)=G(\boldsymbol{u},z)G(\boldsymbol{v},z).
\end{equation}
In addition (cf. \cite[Corollary 2.2]{AdeLek2017.2}), $(\mathcal{G},\times)$ is an abelian group with identity element $\boldsymbol{e}=(e_n)_{n\geq 0}$, where $e_0=1$ and $e_n=0$, $n\in \mathds{N}$.

On the other hand, let $A(x)=(A_n(x))_{n\geq 0}$ be a sequence of polynomials such that $A(0)\in \mathcal{G}$. Recall that $A(x)$ is called an Appell sequence if one of the following equivalent conditions is satisfied
\begin{equation}\label{eq2.3}
A'_n(x)=nA_{n-1}(x),\quad n\in \mathds{N},
\end{equation}
\begin{equation}\label{eq2.4}
A_n(x)=\sum_{k=0}^n \binom{n}{k}A_k(0)x^{n-k},
\end{equation}
or
\begin{equation}\label{eq2.5}
G(A(x),z)=G(A(0),z)e^{xz}.
\end{equation}
Denote by $\mathcal{A}$ the set of all Appell sequences. The binomial convolution of $A(x),C(x)\in \mathcal{A}$, denoted by $(A\times C)(x)= ((A\times C)_n(x))_{n\geq 0}$, is defined as (cf. \cite[Section 3]{AdeLek2017.2})
\begin{equation*}
(A\times C)(x)=A(0)\times C(x)=A(x)\times C(0)=A(0)\times C(0)\times I(x).
\end{equation*}
Equivalently, by
\begin{equation*}
\begin{split}
&(A\times C)_n(x)=\sum_{k=0}^n \binom{n}{k}A_k(0)C_{n-k}(x)=\sum_{k=0}^n \binom{n}{k}C_k(0)A_{n-k}(x)\\
&=\sum_{j_1+j_2+j_3=n}\binom{n}{j_1,j_2,j_3} A_{j_1}(0)C_{j_2}(0)x^{j_3},
\end{split}
\end{equation*}
where
\begin{equation*}
\binom{n}{j_1,\ldots ,j_m}=\dfrac{n!}{j_1!\cdots j_m!},\quad n,j_1,\ldots ,j_m\in \mathds{N}_0,\quad j_1+\cdots +j_m=n,\quad m\in \mathds{N},
\end{equation*}
is the multinomial coefficient. As shown in \cite[Theorem 3.1]{AdeLek2017.2} $(\mathcal{A},\times)$ is an abelian group with identity element $I(x)$. Also, $(A\times C)(x)$ is characterized by its generating function
\begin{equation}\label{eq2.6}
G((A\times C)(x),z)=G(A(0),z)G(C(0),z)e^{xz}.
\end{equation}

The following easy lemma will be used in Sections~\ref{s6} and \ref{s7}.

\begin{lemma}\label{l2.1}
Let $m\in \mathds{N}$ and let $A^{(k)}(x)$, $k=1,\ldots ,m$ be Appell sequences. Let $x_k\in \mathds{R}$, $k=1,\ldots ,m$, with $x_1+\cdots +x_m=x$. Then,
\begin{equation}\label{eq2.7}
A^{(1)}(x_1)\times \cdots \times A^{(m)}(x_m)=(A^{(1)}\times \cdots \times A^{(m)})(x).
\end{equation}
\end{lemma}

\begin{proof}
By \eqref{eq2.2} and \eqref{eq2.6}, both sides in \eqref{eq2.7} have the same generating function given by
\begin{equation*}
G(A^{(1)}(0),z)\cdots G(A^{(m)}(0),z)e^{xz}.
\end{equation*}
Therefore, the conclusion follows from \eqref{eq2.2} and \eqref{eq2.5}.\qed
\end{proof}

\section{The forward difference transformation}\label{s3}

Let $(U_j)_{j\geq 1}$ be a sequence of independent identically distributed random variables having the uniform distribution on $[0,1]$ and denote by
\begin{equation}\label{eq3.8}
S_k=U_1+\cdots +U_k,\qquad k\in \mathds{N}\quad (S_0=0).
\end{equation}
If $f$ is $k$ times differentiable, then the $k$th forward difference of $f$ can be written as (cf. \cite[Lemma 7.2]{AdeLek2017.2})
\begin{equation}\label{eq3.9}
\Delta^k f(x)=\mathds{E}f^{(k)}(x+S_k),\qquad k\in \mathds{N}_0,
\end{equation}
where $\mathds{E}$ stands for mathematical expectation. The following definition will play an important role.

\begin{definition}\label{d3.1}
Let $A(x)\in \mathcal{A}$ and $\boldsymbol{u}\in \mathcal{G}$. We define the forward difference transformation $L_{\boldsymbol{u}} A(x)=(L_{\boldsymbol{u}} A_n(x))_{n\geq 0}$ as
\begin{equation}\label{eq3.10}
L_{\boldsymbol{u}}A_n(x)=\sum_{k=0}^n \dfrac{u_k}{k!}\Delta^k A_n(x)=\sum_{k=0}^n \binom{n}{k}u_k \mathds{E}A_{n-k}(x+S_k).
\end{equation}
\end{definition}

Observe that the second equality in \eqref{eq3.10} follows from \eqref{eq2.3} and \eqref{eq3.9}, since $A_n^{(k)}(x)=n(n-1)\cdots (n-k+1)A_{n-k}(x)$, $k=0,1,\ldots ,n$. This definition can be restated in terms of generating functions as follows.

\begin{theorem}\label{th3.2}
Let $A(x)\in \mathcal{A}$ and $\boldsymbol{u}\in \mathcal{G}$. Then, $L_{\boldsymbol{u}}A(x)$ is an Appell sequence characterized by its generating function
\begin{equation}\label{eq3.11}
G(L_{\boldsymbol{u}}A(x),z)=G(A(0),z)G(\boldsymbol{u},e^{z}-1)e^{xz}.
\end{equation}
\end{theorem}

\begin{proof}
Note that $L_{\boldsymbol{u}}A_0(0)=u_0A_0(0)\neq 0$. Suppose that the radii of $A(0)$ and $\boldsymbol{u}$ are $r>0$ and $s>0$, respectively. Denote by $\rho=\min (r,t)>0$, where $t$ is the unique positive solution to the equation $te^t=s$. Observe that
\begin{equation*}
|G(\boldsymbol{u},e^z-1)|\leq G(|\boldsymbol{u}|,|z|e^{|z|})\leq G(|\boldsymbol{u}|,s)<\infty,\quad |z|\leq t.
\end{equation*}
This means that formula \eqref{eq3.11} makes sense for $|z|\leq \rho$. For $|z|\leq \rho$, we have
\begin{equation}\label{eq3.12}
\begin{split}
&G(L_{\boldsymbol{u}}A(x),z)=\sum_{n=0}^\infty \dfrac{z^n}{n!} \sum_{k=0}^n \dfrac{u_k}{k!}\Delta^k A_n(x)=\sum_{k=0}^\infty \dfrac{u_k}{k!}\sum_{n=k}^\infty \Delta^k A_n(x)\dfrac{z^n}{n!}\\
&=\sum_{k=0}^\infty \dfrac{u_k}{k!}\Delta^k \left ( \sum_{n=0}^\infty A_n(x)\dfrac{z^n}{n!}\right )=\sum_{k=0}^\infty \dfrac{u_k}{k!}\Delta^k G(A(x),z),
\end{split}
\end{equation}
because $\Delta^kA_n(x)=0$, $0\leq n<k$. On the other hand, we have from \eqref{eq2.5}
\begin{equation*}
\Delta^k G(A(x),z)=G(A(0),z)e^{xz}(e^z-1)^k,\quad k\in \mathds{N}_0.
\end{equation*}
This, together with \eqref{eq3.12}, shows \eqref{eq3.11}. In turn, formula \eqref{eq3.11} implies that
\begin{equation*}
G(L_{\boldsymbol{u}}A(x),z)=G(L_{\boldsymbol{u}}A(0),z)e^{xz},
\end{equation*}
which, by virtue of \eqref{eq2.5}, shows the result.\qed
\end{proof}

Theorem~\ref{th3.2}  allows us to prove many identities for Appell sequences in a very easy way, as shown in the following result.

\begin{corollary}\label{cor3.3}
Let $A(x),C(x)\in \mathcal{A}$ and let $\boldsymbol{u},\boldsymbol{v}\in \mathcal{G}$. Then,
\begin{equation*}
\begin{split}
L_{\boldsymbol{u}}L_{\boldsymbol{v}}(A\times C)(x)&=L_{\boldsymbol{v}}L_{\boldsymbol{u}}(A\times C)(x)= L_{\boldsymbol{u}\times \boldsymbol{v}}(A\times C)(x)\\
&=(L_{\boldsymbol{u}}A\times L_{\boldsymbol{v}}C)(x)=(L_{\boldsymbol{v}}A\times L_{\boldsymbol{u}}C)(x).
\end{split}
\end{equation*}
\end{corollary}

\begin{proof}
Using \eqref{eq2.6} and \eqref{eq3.11}, it can be checked that the common generating function of all Appell sequences under consideration is
\begin{equation*}
G(A(0),z)G(C(0),z) G(\boldsymbol{u},e^z-1)G(\boldsymbol{v},e^z-1)e^{xz}.
\end{equation*}
Therefore, the result follows from \eqref{eq2.5}.\qed
\end{proof}

\begin{corollary}\label{cor3.4}
Let $\boldsymbol{u}\in \mathcal{G}$. Then, the map $L_{\boldsymbol{u}}:\mathcal{A}\to \mathcal{A}$ is biyective and $L_{\boldsymbol{u}}^{-1}=L_{\boldsymbol{v}}$, where $\boldsymbol{v}$ is the inverse of $\boldsymbol{u}$ in $(\mathcal{G},\times)$.

Moreover, $L_{\boldsymbol{u}}$ is a multiplier, i.e.,
\begin{equation}\label{eq3.13}
L_{\boldsymbol{u}}(A\times C)(x)=(A\times L_{\boldsymbol{u}} C)(x)=(L_{\boldsymbol{u}}A\times C)(x),\quad A(x),C(x)\in \mathcal{A}.
\end{equation}
\end{corollary}

\begin{proof}
Suppose that $L_{\boldsymbol{u}}A(x)=L_{\boldsymbol{u}} C(x)$. From \eqref{eq3.11}, we see that $G(A(0),z)=G(C(0),z)$, which implies that $A(0)=C(0)$ and, a fortiori, $A(x)=C(x)$. On the other hand, let $A(x)\in \mathcal{A}$ and let $\boldsymbol{v}$ be the inverse of $\boldsymbol{u}$ in $(\mathcal{G},\times)$. Choosing $C(x)=I(x)$ in Corollary~\ref{cor3.3}, we see that $L_{\boldsymbol{u}}(L_{\boldsymbol{v}}A)(x)=L_{\boldsymbol{v}}(L_{\boldsymbol{u}}A)(x)=L_{\boldsymbol{e}}A(x)=A(x)$, thus showing that $L_{\boldsymbol{u}}$ is biyective.

Finally, statement \eqref{eq3.13} follows from Corollary~\ref{cor3.3} by choosing $\boldsymbol{v}=\boldsymbol{e}$. The proof is complete.\qed
\end{proof}

Fix $\boldsymbol{u}\in \mathcal{G}$. Setting $C=I$ in Corollary~\ref{cor3.4}, we see that
\begin{equation*}
    L_{\boldsymbol{u}}A(x)=(A\times L_{\boldsymbol{u}}I)(x),\quad A\in \mathcal{A}.
\end{equation*}
In other words, the map $L_{\boldsymbol{u}}$ is determined by the (multiplier) element $L_{\boldsymbol{u}}I(x)$.

\section{Stirling numbers}\label{s4}

Recall that the Stirling numbers of the first and second kind, respectively denoted by $s(n,k)$ and $S(n,k)$, $k=0,1,\ldots ,n$, are defined by (see, for instance, Abramowitz and Stegun \cite[p. 824]{AbrSte1992} or Roman \cite[p. 60]{Rom1984})
\begin{equation}\label{eq4.13}
(x)_n=\sum_{k=0}^n s(n,k) x^k,\qquad x^n=\sum_{k=0}^n S(n,k)(x)_k,
\end{equation}
where $(x)_n=x(x-1)\cdots (x-n+1)$. Equivalently, the Stirling numbers are defined by their generating functions
\begin{equation}\label{eq4.14}
\dfrac{\log^k (z+1)}{k!}=\sum_{n=k}^\infty s(n,k) \dfrac{z^n}{n!},\quad \dfrac{(e^z-1)^k}{k!}=\sum_{n=k}^\infty S(n,k) \dfrac{z^n}{n!},\quad |z|<1.
\end{equation}

Theorem~\ref{th3.2} suggests that the Stirling numbers play a fundamental role in order to obtain explicit expressions for Appell sequences (see also Theorem~\ref{th5.7} below and the examples in Sections~\ref{s6} and \ref{s7}). We start with a probabilistic representation of such numbers which will be useful to deal with specific examples.

Let $(T_j)_{j\geq 1}$ be a sequence of independent identically distributed random variables having the exponential density
\begin{equation}\label{eq4.15}
\rho(\theta)=e^{-\theta},\quad \theta \geq 0.
\end{equation}
We assume that $(T_j)_{j\geq 1}$ and $(U_j)_{j\geq 1}$, as given in \eqref{eq3.8}, are mutually independent. Denote by
\begin{equation}\label{eq4.16}
S_k^\star=U_1T_1+\cdots +U_kT_k,\quad k\in \mathds{N}\qquad (S_0^\star=0).
\end{equation}
Observe that
\begin{equation}\label{eq4.19.2}
\mathds{E}U_1^n=\dfrac{1}{n+1},\qquad \mathds{E}T_1^n=n!.
\end{equation}

The following auxiliary result is known.

\begin{lemma}\label{l4.5}
For any $n\in \mathds{N}_0$ and $k=0,1,\ldots ,n$, we have
\begin{equation*}
s(n,k)=(-1)^{n-k} \binom{n}{k} \mathds{E} (S_k^\star)^{n-k},\qquad S(n,k)=\binom{n}{k}\mathds{E}S_k^{n-k}.
\end{equation*}
\end{lemma}

The probabilistic representations for $s(n,k)$ and $S(n,k)$ given in Lemma~\ref{l4.5} are shown in \cite{AdeLek2017.3} and \cite{Sun2005}, respectively. On the other hand, choosing $f(x)=I_n(x)$ in \eqref{eq3.9}, we have from Lemma~\ref{l4.5}
\begin{equation}\label{eq4.19}
S(n,k)=\dfrac{\Delta^k I_n(0)}{k!},\quad k=0,1,\ldots ,n.
\end{equation}

Let $\boldsymbol{u}$ and $\boldsymbol{v}$ be two real sequences defined by the equivalent equations (cf. Comtet \cite[p. 144]{Com1974})
\begin{equation}\label{eq4.20}
u_n=\sum_{k=0}^n s(n,k)v_k \quad\Leftrightarrow \quad v_n=\sum_{k=0}^n S(n,k) u_k.
\end{equation}

The following simple auxiliary result will be very useful.

\begin{lemma}\label{l4.6}
Let $\boldsymbol{u}$ and $\boldsymbol{v}$ be as in \eqref{eq4.20}. Then, $\boldsymbol{u}\in \mathcal{G}$ if and only if $\boldsymbol{v}\in \mathcal{G}$. In such a case, we have
\begin{equation}\label{eq4.21}
G(\boldsymbol{v},z)=G(\boldsymbol{u},e^z-1),\quad G(\boldsymbol{u},z)=G(\boldsymbol{v},\log (z+1)).
\end{equation}
\end{lemma}

\begin{proof}
Observe that $u_0=v_0\neq 0$. Suppose that $\boldsymbol{u}\in \mathcal{G}$ with radius $r>0$. For $|e^z-1|\leq r$, we have from \eqref{eq4.14} and \eqref{eq4.20}
\begin{equation*}
G(\boldsymbol{u},e^z-1)=\sum_{k=0}^\infty u_k \sum_{n=k}^\infty S(n,k)\dfrac{z^n}{n!}=\sum_{n=0}^\infty \dfrac{z^n}{n!}\sum_{k=0}^n S(n,k) u_k=G(\boldsymbol{v},z).
\end{equation*}
The second equality in \eqref{eq4.21} is similarly shown. The proof is complete.\qed
\end{proof}

\section{Main results}\label{s5}

The following result characterizes Appell sequences in terms of forward difference transformations of the identity.

\begin{theorem}\label{th5.7}
$A(x)\in \mathcal{A}$ if and only if there exists a unique $\boldsymbol{a}\in \mathcal{G}$ such that
\begin{equation}\label{eq5.22}
A(x)=L_{\boldsymbol{a}}I(x).
\end{equation}
Such $\boldsymbol{a}\in \mathcal{G}$, called the sequence associated to $A(x)$, is given by the equivalent equations
\begin{equation}\label{eq5.23}
a_n=\sum_{k=0}^n s(n,k) A_k(0)\quad \Leftrightarrow \quad A_n(0)=\sum_{k=0}^n S(n,k) a_k.
\end{equation}

As a consequence, we have the explicit expressions
\begin{equation}\label{eq5.24}
A_n(x)=\sum_{k=0}^n \dfrac{a_k}{k!}\Delta^k I_n(x)=\sum_{k=0}^n \binom{n}{k}a_k \mathds{E}(x+S_k)^{n-k}.
\end{equation}
\end{theorem}

\begin{proof}
Let $A(x)\in \mathcal{A}$. We pose the equation
\begin{equation*}
G(A(x),z)=G(L_{\boldsymbol{a}}I(x),z).
\end{equation*}
From \eqref{eq2.5} and \eqref{eq3.11}, this equation has the form
\begin{equation}\label{eq5.25}
G(A(0),z)=G(\boldsymbol{a},e^z-1).
\end{equation}
Therefore, by \eqref{eq4.20} and Lemma~\ref{l4.6}, equation \eqref{eq5.25} has a unique solution $\boldsymbol{a}\in \mathcal{G}$ given by \eqref{eq5.23}. Formula \eqref{eq5.24} readily follows from \eqref{eq5.22} and Definition~\ref{d3.1}. The proof is complete.\qed
\end{proof}

Dealing with specific examples, the interesting question is to find a closed form expression of the associated sequence $\boldsymbol{a}$. In this respect, the behaviour of such associated sequences under certain transformations of Appell polynomials is described in the following two results.

\begin{theorem}\label{th5.8}
Let $A(x)$ and $C(x)$ be Appell polynomials with associated sequences $\boldsymbol{a}$ and $\boldsymbol{c}$, respectively. Then,
\begin{enumerate}[(a)]
\item $(A\times C)(x)$ has associated sequence $\boldsymbol{a}\times \boldsymbol{c}$.
\item The inverse of $A(x)$ in $(\mathcal{A},\times)$ has the inverse of $\boldsymbol{a}$ in $(\mathcal{G},\times)$ as associated sequence.
\item For any $\boldsymbol{b}\in \mathcal{G}$, $L_{\boldsymbol{b}}A(x)$ has associated sequence $\boldsymbol{a}\times \boldsymbol{b}$.
\end{enumerate}
\end{theorem}

\begin{proof}
By \eqref{eq5.22} and Corollary~\ref{cor3.3}, we have
\begin{equation*}
(A\times C)(x)=(L_{\boldsymbol{a}}I\times L_{\boldsymbol{c}}I)(x)=L_{\boldsymbol{a}\times \boldsymbol{c}}I(x),
\end{equation*}
which shows part (a). To show (b), let $C(x)$ be the inverse of $A(x)$ with associated sequence $\boldsymbol{c}$. Again by Corollary~\ref{cor3.3}, we have
\begin{equation*}
I(x)=(A\times C)(x)=(L_{\boldsymbol{a}}I\times L_{\boldsymbol{c}}I)(x)=L_{\boldsymbol{a}\times \boldsymbol{c}}I(x),
\end{equation*}
which implies that $\boldsymbol{a}\times \boldsymbol{c}=\boldsymbol{e}$. Finally, part (c) follows from \eqref{eq5.22} and Corollary~\ref{cor3.3} with $C=I$.\qed
\end{proof}

Let $Y$ be a random variable such that
\begin{equation}\label{eq5.33}
Ee^{r|Y|}<\infty,
\end{equation}
for some $r>0$. In \cite{AdeLek2017.2}, we have considered the following transformation based on expectations. Given $A(x)\in \mathcal{A}$, we define the Appell sequence $R_YA(x)=(R_YA_n(x))_{n\geq 0}$ as
\begin{equation}\label{eq5.34}
  R_YA_n(x)=\mathds{E}A_n(x+Y)=\sum_{k=0}^n \binom{n}{k}A_k(0)\mathds{E}(x+Y)^{n-k}.
\end{equation}
The transformation $R_Y:\mathcal{A}\to \mathcal{A}$ is actually a multiplier (c.f. \cite[Proposition 5.1]{AdeLek2017.2}). The transformation of the identity $I(x)$ is given by
\begin{equation}\label{eq5.35}
  R_YI(x)=(\mathds{E}(x+Y)^n)_{n\geq 0}.
\end{equation}
Observe that, for any $A(x)\in \mathcal{A}$, we have from \eqref{eq5.34} and \eqref{eq5.35}
\begin{equation}\label{eq5.36}
  R_YA(x)= (A\times R_YI)(x).
\end{equation}
It turns out (cf. \cite[Proposition 5.1]{AdeLek2017.2}) that $R_YA(x)$ is an Appell sequence characterized by its generating function
\begin{equation}\label{eq5.37}
  G(R_YA(x),z)=G(A(0),z)\mathds{E}e^{z(x+Y)}.
\end{equation}
With these notations, we state the following.

\begin{theorem}\label{th5.9}
The Appell sequence $R_YI(x)$ has associated sequence $\boldsymbol{y}=(y_n)_{n\geq 0}$ given by $y_n=\mathds{E}(Y)_n$. As a consequence,
\begin{equation}\label{eq5.38}
  R_YI_n(x)=\mathds{E}(x+Y)^n=\sum_{k=0}^n \mathds{E}\binom{Y}{k}\Delta^kI_n(x).
\end{equation}

Moreover, if $A(x)\in \mathcal{A}$ has associated sequence $\boldsymbol{a}$, then $R_YA(x)$ has associated sequence $\boldsymbol{a}\times \boldsymbol{y}$.
\end{theorem}

\begin{proof}
Let $z\in \mathds{C}$ with $|e^z-1|<1$. Using the binomial expansion, we get
\begin{equation*}
  \mathds{E}e^{zY}=\mathds{E}(e^z-1+1)^Y=\sum_{n=0}^\infty \mathds{E}\binom{Y}{n}(e^z-1)^n=G(\boldsymbol{y},e^z-1).
\end{equation*}
This implies, by virtue of Theorem~\ref{th3.2} and \eqref{eq5.37}, that $R_YI(x)$ and $L_{\boldsymbol{y}}I(x)$ have the same generating function. Therefore, $R_YI(x)=L_{\boldsymbol{y}}I(x)$, thus showing \eqref{eq5.38}. The final statement in Theorem~\ref{th5.9} follows from Theorem~\ref{th5.8}(a) and \eqref{eq5.36}. The proof is complete.\qed
\end{proof}

\begin{remark}\label{rem5.4}
Choosing $Y=y\in \mathds{R}$ in \eqref{eq5.38}, we obtain
\begin{equation*}
  (x+y)^n=\sum_{k=0}^n \binom{y}{k}\Delta^k I_n(x).
\end{equation*}
In view of \eqref{eq4.13} and \eqref{eq4.19}, this formula can be seen as a generalization of the classical identity defining the Stirling numbers $S(n,k)$ of the second kind.
\end{remark}

\section{Generalized Bernoulli polynomials of real order}\label{s6}

Let $Y$ be a random variable such that
\begin{equation}\label{eq6.35.1}
  r\mathds{E}|Y|e^{r|Y|}<1,
\end{equation}
for some $r>0$. Observe that condition \eqref{eq6.35.1} is more restrictive than condition \eqref{eq5.33}. Also, let $(Y_j)_{j\geq 1}$ be a sequence of independent copies of $Y$ and denote by
\begin{equation*}
  W_k=Y_1+\cdots + Y_k,\qquad k\in \mathds{N},\quad  (W_0=0).
  \end{equation*}
In a recent paper (cf. \cite{AdeLek2018}), we have shown that for any real $t$ we have
\begin{equation}\label{eq6.35.2}
  (\mathds{E}e^{zY})^t=G(\boldsymbol{y(t)},z),
\end{equation}
where $\boldsymbol{y(t)}=(y_n(t))_{n\geq 0}\in \mathcal{G}$ is given by
\begin{equation}\label{eq6.35.3}
  y_n(t)=\sum_{k=0}^n \binom{t}{k} \binom{n-t}{n-k}\mathds{E}W_k^n.
\end{equation}
Note that if $t=m\in \mathds{N}$, formula \eqref{eq6.35.3} takes on the simple form
\begin{equation}\label{eq6.35.4}
  y_n(m)=\mathds{E}W_m^n.
\end{equation}
On the other hand, let $t\in \mathds{R}$. The generalized Bernoulli polynomials $B(t;x)$ of order $t$ are defined by means of the generating function
\begin{equation}\label{eq6.35.5}
  G(B(t;x),z)=\left ( \dfrac{z}{e^z-1}\right )^t e^{xz}.
\end{equation}
For $t=1$, $B(x):=B(1;x)$ are the classical Bernoully polynomials. We give the following closed form expression for $B(t;x)$.

\begin{theorem}\label{th6.1}
Let $t\in \mathds{R}$. Then, $B(t;x)$ has associated sequence $\boldsymbol{b(t)}=(b_n(t))_{n\geq 0}$ given by
\begin{equation}\label{eq6.35.6}
  b_n(t)=\sum_{k=0}^n \dfrac{\binom{t}{k} \binom{n-t}{n-k}}{\binom{k+n}{n}}\, s(k+n,k),
\end{equation}
$s(n,k)$ being the Stirling numbers of the first kind.

In particular, for $t=m\in \mathds{N}$, we have
\begin{equation}\label{eq6.35.7}
  b_n(m)= \dfrac{s(m+n,m)}{\binom{m+n}{n}}.
\end{equation}
\end{theorem}

\begin{proof}
Let $U$ and $T$ be two independent random variables such that $U$ is uniformly distributed on $[0,1]$ and $T$ has the exponential density given in \eqref{eq4.15}. Note that
\begin{equation*}
  \mathds{E}e^{-zUT}=\mathds{E}\dfrac{1}{1+zU}=\dfrac{\log (1+z)}{z},\qquad |z|<1,
\end{equation*}
thus implying that
\begin{equation}\label{eq6.35.8}
  \mathds{E}e^{-(e^z-1)UT}=\dfrac{z}{e^z-1},\qquad |e^z-1|<1.
\end{equation}
Assume that $|e^z-1|<1$. Choosing $Y=UT$ in \eqref{eq6.35.2} and \eqref{eq6.35.3} and taking into account \eqref{eq6.35.5} and \eqref{eq6.35.8}, we have
\begin{equation}\label{eq6.35.9}
  G(B(t;x),z)=\left (\mathds{E}e^{-(e^z-1)UT} \right )^te^{xz}=G(\boldsymbol{b(t)},e^z-1)e^{xz},
\end{equation}
where
\begin{equation}\label{eq6.35.10}
  b_n(t)=\sum_{k=0}^n \binom{t}{k}\binom{n-t}{n-k}\mathds{E}(-S^\star_k)^n=\sum_{k=0}^n \dfrac{\binom{t}{k}\binom{n-t}{n-k}}{\binom{k+n}{n}}s(k+n,k),
\end{equation}
thanks to \eqref{eq4.16} and Lemma~\ref{l4.5}. Therefore, the first statement in Theorem~\ref{th6.1} follows from Theorem~\ref{th3.2}, \eqref{eq6.35.9}, and \eqref{eq6.35.10}.

Finally, if $t=m\in \mathds{N}$, we have from \eqref{eq6.35.10} and Lemma~\ref{l4.5}
\begin{equation}\label{eq6.35.11}
  b_n(m)=\mathds{E}(-S_m^\star)^n= \dfrac{s(m+n,m)}{\binom{m+n}{m}}.
\end{equation}
This shows \eqref{eq6.35.7} and completes the proof.\qed
\end{proof}

The numbers $(b_n(m))_{n\geq 0}$ defined in \eqref{eq6.35.11} are called Daehee numbers of the first kind of order $m$ by Kim \textit{et al.} \cite{KimKimLeeSeo2014}. A similar result to that in Theorem~\ref{th6.1} can be found in the recent paper by Boutiche \textit{et al.} \cite{BouRahSri2017}.

Denote by
\begin{equation*}
H_k=\sum_{j=1}^k \dfrac{1}{j},\quad k\in \mathds{N},
\end{equation*}
the $k$th harmonic number.

\begin{theorem}\label{th6.11}
Let $m\in \mathds{N}$ and let $x_1,\ldots ,x_m\in \mathds{R}$ with $x_1+\cdots +x_m=x$. Then,
\begin{equation}\label{eq6.36}
\begin{split}
&\sum_{j_1+\cdots +j_m=n}\binom{n}{j_1,\ldots, j_m}B_{j_1}(x_1)\cdots B_{j_m}(x_m)=B_n(m;x)\\
&=\sum_{k=0}^n \dfrac{s(m+k,m)}{k! \binom{m+k}{m}}\Delta^k I_n(x).
\end{split}
\end{equation}

In addition, we have for $m=2$
\begin{equation}\label{eq6.37}
\sum_{k=0}^n \binom{n}{k}B_k(x_1)B_{n-k}(x_2)=B_n(2;x)=2\sum_{k=0}^n \dfrac{(-1)^kH_{k+1}}{k+2}\Delta^k I_n(x).
\end{equation}
\end{theorem}

\begin{proof}
As follows from \eqref{eq2.6} and \eqref{eq6.35.5}, we see that
\begin{equation}\label{eq6.38}
  G(B(m;x),z)=G((\stackrel{\stackrel{m}{\smile}}{B\times \cdots \times B})(x),z),
\end{equation}
which implies, by virtue of Lemma~\ref{l2.1}, that
\begin{equation*}
B(x_1)\times\cdots \times B(x_m)=(\stackrel{\stackrel{m}{\smile}}{B\times \cdots \times B})(x)=B(m;x).
\end{equation*}
This shows the first equality in \eqref{eq6.36}. On the other hand, Theorem~\ref{th6.1} states that $B(m;x)$ has associated sequence $\boldsymbol{b(m)}$ given by \eqref{eq6.35.11}. Therefore, the second equality in \eqref{eq6.36} follows from \eqref{eq5.24}.

For $m=2$, we have from \eqref{eq4.16}, \eqref{eq4.19.2} and \eqref{eq6.35.11}
\begin{equation*}
\begin{split}
&b_n(2)=(-1)^n \mathds{E}(U_1T_1+U_2T_2)^n=(-1)^n\sum_{j=0}^n \binom{n}{j}\mathds{E}U_1^j \mathds{E}T_1^j\mathds{E}U_2^{n-j}\mathds{E}T_2^{n-j}\\
&=(-1)^n n! \sum_{j=0}^n \dfrac{1}{(j+1)(n-j+1)}=\dfrac{(-1)^n n!}{n+2}\,2 H_{n+1}.
\end{split}
\end{equation*}
As above, this shows the second equality in \eqref{eq6.37} and completes the proof.\qed
\end{proof}

Setting $m=1$ in \eqref{eq6.36}, we obtain the explicit formula for the Bernoulli polynomials $B(x)$ mentioned in \eqref{eq1.1}. Also, choosing $x_1=\cdots =x_m=0$ in \eqref{eq6.36}, we obtain
\begin{equation*}
\sum_{j_1+\cdots +j_m=n}\binom{n}{j_1,\cdots, j_m}B_{j_1}(0)\cdots B_{j_m}(0)=\sum_{k=0}^n \dfrac{s(m+k,m)S(n,k)}{\binom{m+k}{m}},
\end{equation*}
as follows from \eqref{eq4.19}. Higher order convolution identities for Bernoulli polynomials were first given by Dilcher \cite{Dil1996}. For degenerate Bernoulli polynomials and for Apostol-Euler polynomials, we refer the reader to Wu and Pan \cite{WuPan2014} and He and Araci \cite{HeAra2014}, respectively. Wang \cite{Wan2013}, Agoh and Dilcher \cite{AgoDil2014}, He \cite{He2017}, and Dilcher and Vignat \cite{DilVig2016}, have obtained convolution identities with the multinomial coefficients replaced by other coefficients. The distinctive feature of Theorem~\ref{th6.11} with respect to the aforementioned results is that the right-hand sides in identities \eqref{eq6.36} and \eqref{eq6.37} do not contain the Bernoulli polynomials themselves.

\section{The generalized Apostol-Euler polynomials of real order}\label{s7}

Let $t\in \mathds{R}$ and $0\leq \beta \leq 1$. We consider the generalized Apostol-Euler polynomials $E(t,\beta;x)$ defined via their generating function as
\begin{equation}\label{eq7.43}
G(E(t,\beta;x),z)=\dfrac{e^{xz}}{(1+\beta (e^z-1))^t}.
\end{equation}
The Apostol-Euler polynomials and the classical Euler polynomials are defined, respectively, as
\begin{equation*}
E(\beta;x)=E(1,\beta;x),\qquad E(x)=E(1/2;x).
\end{equation*}
The particular form of the generating function in \eqref{eq7.43} leads us to an immediate application of Theorems~\ref{th3.2} and \ref{th5.7}.

\begin{theorem}\label{th7.1}
Let $t\in \mathds{R}$ and $0\leq \beta \leq 1$. Then, $E(t,\beta;x)$ has associated sequence $\boldsymbol{a(t)}=(a_n(t))_{n\geq 0}$ given by
\begin{equation}\label{eq7.42.1}
  a_n(t)=(-t)_n\beta^n.
\end{equation}
\end{theorem}

\begin{proof}
Suppose that $|e^z-1|<1/\beta$. Using the binomial expansion in \eqref{eq7.43}, we see that
\begin{equation*}
  G(E(t,\beta;x),z)=G(\boldsymbol{a(t)},e^z-1)e^{xz},
\end{equation*}
where $\boldsymbol{a(t)}\in \mathcal{G}$ is defined in \eqref{eq7.42.1}. Thus, the conclusion follows from Theorems~\ref{th3.2} and \ref{th5.7}.\qed
\end{proof}

We mention that Theorem~\ref{th7.1} has recently been obtained by Boutiche \textit{et al.} \cite{BouRahSri2017} for the generalized Euler polynomials $E(t,1/2;x)$ by using the generalized Stirling numbers $S_n^k(x)$ defined in \eqref{eq1.0.4}.

\begin{theorem}\label{th7.13}
Let $m\in \mathds{N}$ and $x_1,\ldots ,x_m \in \mathds{R}$ with $x_1+\cdots +x_m=x$. Then,
\begin{equation}\label{eq7.45}
\begin{split}
&\sum_{j_1+\cdots +j_m=n} \binom{n}{j_1,\ldots ,j_m} E_{j_1}(\beta; x_1)\cdots E_{j_m}(\beta ; x_m)\\
&=E_n(m,\beta; x)=\sum_{k=0}^n \binom{-m}{k}\beta^k \Delta^k I_n(x).
\end{split}
\end{equation}
\end{theorem}

\begin{proof}
It follows from \eqref{eq2.6} and \eqref{eq7.43} that
\begin{equation}\label{eq7.44}
(\stackrel{\stackrel{m}{\smile}}{E(\beta; \cdot)\times \cdots \times E(\beta; \cdot)})(x)=E(m,\beta; x).
\end{equation}
By Lemma~\ref{l2.1}, this implies that
\begin{equation*}
\stackrel{\stackrel{m}{\smile}}{E(\beta; x_1)\times \cdots \times E(\beta; x_m)}=E(m,\beta; x).
\end{equation*}
which, in conjunction with Theorem~\ref{th7.1}, shows the result.\qed
\end{proof}

Similar identities to that in Theorem~\ref{th7.13} were obtained by Dilcher \cite{Dil1996}, He and Araci \cite{HeAra2014}, and Dilcher and Vignat \cite{DilVig2016}(for degenerate Euler polynomials, we refer the reader to Wu and Pan \cite{WuPan2014}). As in Theorem~\ref{th6.11}, note that the right-hand side in formula \eqref{eq7.45} does not contain the Apostol-Euler polynomials.

Applying Theorem~\ref{th7.13} with $m=1$ and $\beta=1/2$, we obtain the explicit formula for the Euler polynomials stated in \eqref{eq1.1}.

\begin{theorem}\label{th7.14}
Let $m,r\in \mathds{N}$ and let $x_1,\ldots ,x_{m+r}\in \mathds{R}$ with $x_1+\cdots +x_{m+r}=x$. Then,
\begin{equation}\label{eq7.42.2}
  \begin{split}
     & \sum_{j_1+\cdots +j_{m+r}=n} \binom{n}{j_1,\cdots, j_{m+r}}B_{j_1}(x_1)\cdots B_{j_m}(x_m)E_{j_{m+1}}(\beta; x_{m+1})\cdots  E_{j_{m+r}}(\beta; x_{m+r})\\
      & = \left ( B_n(m,\cdot)\times E(r,\beta;\cdot)\right )(x)=\sum_{k=0}^n v_k \Delta^k I_n(x),
  \end{split}
\end{equation}
where
\begin{equation}\label{eq7.42.3}
  v_k=m!\sum_{j=0}^k \binom{-r}{j}\beta^j \dfrac{s(m+k-j,m)}{(m+k-j)!}.
\end{equation}
\end{theorem}

\begin{proof}
Using \eqref{eq6.38}, \eqref{eq7.44} and Lemma~\ref{l2.1}, we see that
\begin{equation*}
\begin{split}
  & B(x_1)\times \cdots \times B(x_m)\times E(\beta; x_{m+1})\times \cdots \times E(\beta; x_{m+r}) \\
    & =B(m,x_1+\cdots +x_m)\times E(r,\beta; x_{m+1}+\cdots +x_{m+r})\\
    & = \left ( B(m;\cdot)\times E(r,\beta;\cdot) \right ) (x),
\end{split}
\end{equation*}
thus showing the first equality in \eqref{eq7.42.2}. By Theorems~\ref{th6.1} and \ref{th7.1}, the Appell sequences $B(m;x)$ and $E(r,\beta;x)$ have associated sequences $\boldsymbol{b(m)}$ and $\boldsymbol{a(r)}$ defined in \eqref{eq6.35.7} and \eqref{eq7.42.1}, respectively. We therefore have from Theorem~\ref{th5.8}(a)
\begin{equation}\label{eq7.42.4}
  \left (B_n(m;\cdot)\times E_n(r,\beta;\cdot) \right )(x)=\sum_{k=0}^n \dfrac{(a(r)\times b(m))_k}{k!}\Delta^k I_n(x).
\end{equation}
Again by \eqref{eq6.35.7} and \eqref{eq7.42.1}, we have
\begin{equation*}
  \dfrac{1}{k!}(a(r)\times b(m))_k=\dfrac{1}{k!} \sum_{j=0}^k \binom{k}{j}a_j(r)b_{k-j}(m)=v_k,
\end{equation*}
as follows from \eqref{eq7.42.3}. This, together with \eqref{eq7.42.4}, shows the second equality in \eqref{eq7.42.2} and concludes the proof.\qed
\end{proof}

Wang \cite{Wan2013} obtained analogous identities to that in Theorem~\ref{th7.14} with the multinomial coefficients replaced by other coefficients, whereas He and Araci \cite{HeAra2014} considered sums of form \eqref{eq7.42.2} where the Bernoulli polynomials are replaced by Apostol-Bernoulli polynomials. Finally, observe that the right-hand side in \eqref{eq7.42.2} does not contain neither the polynomials $B_n(x)$ nor $E_n(\beta;x)$.

\begin{acknowledgements}
The authors are partially supported by Research Projects DGA (E-64), MTM2015-67006-P, and by FEDER funds.
\end{acknowledgements}

\bibliographystyle{spmpsci}      

 \bibliography{AdellLekuonaArXiv2018R}

\end{document}